\documentclass{amsart}
\usepackage{amsmath,amsfonts,amsthm}
\usepackage{amssymb,latexsym}
\usepackage{graphics}
\usepackage[colorlinks]{hyperref}

\usepackage[all]{xypic}
\usepackage{amssymb}

\theoremstyle{plain}
\theoremstyle{definition}
\newtheorem{theorem}{Theorem}[section]
\newtheorem{lemma}[theorem]{Lemma}

\newtheorem{corollary}[theorem]{Corollary}

\newtheorem{definition}[theorem]{Definition}

\newtheorem{counterexample}[theorem]{Counterexample}

\newtheorem{note}[theorem]{Note}

\newtheorem{convention}[theorem]{Convention}
\newtheorem{remark}[theorem]{Remark}
\theoremstyle{remark}

\numberwithin{equation}{section}
\newcommand{\SP}{\: \: \: \: \:}

\title[Li--Yorke and Devaney chaotic weighted generalized shifts]{Li--Yorke and Devaney chaotic uniform dynamical systems
amongst weighted shifts}
\author[F. Ayatollah Zadeh Shirazi, E. Hakimi, A. Hosseini, R. Rezavand]{Fatemah Ayatollah Zadeh Shirazi, Elaheh Hakimi, \\ Arezoo Hosseini, Reza Rezavand}
\begin{document}
\begin{abstract}
In this paper, for finite discrete field $F$, nonempty set
$\Gamma$, weight vector $\mathfrak{w}=({\mathfrak
w}_\alpha)_{\alpha\in\Gamma}\in F^\Gamma$ and weighted
generalized shift $\sigma_{\varphi,{\mathfrak w}}:F^\Gamma\to
F^\Gamma$, we find necessary and sufficient conditions for
uniform dynamical system $(F^\Gamma,\sigma_{\varphi,{\mathfrak
w}})$ to be Li--Yorke chaotic. Next we find necessary and
sufficient conditions for $(F^\Gamma,\sigma_{\varphi,{\mathfrak
w}})$ to be Devaney chaotic.
\end{abstract}
\maketitle
\noindent {\small {\bf 2020 Mathematics Subject Classification:}  37B02, 54H15 \\
{\bf Keywords:}} Devaney chaotic, Li--Yorke chaotic, sensitive,
topological transitive, weighted generalized shift.
\section{Introduction}
\noindent By a dynamical system $(X,f)$ we mean a topological
space $X$ and a continuous map $f:X\to X$. The family of all
metric dynamical systems is a subclass of the family of all
uniform dynamical systems which is in its own turn a subclass of
the family of all uniform transformation semigroups. Hence it is
natural for ideas employed in metric  dynamical systems to be
extended to uniform  dynamical systems or for ideas adopted in
uniform transformation semigroups to be restricted to uniform
dynamical systems. Amongst several properties introduced for
uniform transformation semigroups we want to make focus on
Devaney and Li--Yorke chaos. In this text we study Devaney and
Li--Yorke chaos in a sub--class of uniform dynamical systems
``weighted generalized shifts''.
\\
In classical form, for compact metric space $(X,d)$, we call the dynamical system $(X,f)$  ``Devaney'' chaotic if
it has the following properties~\cite{de, wang}:
\begin{itemize}
\item DPP (Dense Periodic Points). $Per(f)\:(=\{x\in X:\exists n\geq1\: f^n(x)=x\})$ is a dense subset of $X$,
\item TT (Topological Transitive). For all nonempty open subsets $U,V$ of $X$ there exists $n\geq1$ such that $f^n(U)\cap V\neq\varnothing$,
\item SIC (Sensitive to Initial Conditions). There exists $\delta>0$ such that for all $x\in X$ and open neighbourhood $U$ of $x$ there exists $y\in U$ and $n\geq0$ such that $d(f^n(x),f^n(y))>\delta$.
\end{itemize}
Several authors have tried to extend the above concept of Devaney chaos to uniform dynamical systems$\setminus$transformation (semi-)groups by redefining (see e.g.~\cite{baz, jaleb}):
\begin{itemize}
\item sensitivity in uniform dynamical systems$\setminus$transformation (semi-)groups and
\item periodic points in uniform transformation (semi-)groups.
\end{itemize}
Let us see what about the other well--known chaos, ``Li--Yorke''.
In traditional form a metric dynamical system is Li--Yorke chaotic
if it contains an uncountable scrambled
subset~\cite{li, li-sen},
also one can extend the above concept of a Li--Yorke chaotic
dynamical system to Li--Yorke chaotic transformation (semi-)groups with special considerations on
\\ $\bullet$ phase (semi-)group (for finitely generated case of phase semigroup see \cite{chu}), or
\\ $\bullet$ phase space (uniform phase space)
\\
or even both of the above evaluations. One may consider even Li--Yorke chaotic uniform transformation semigroups modulo
an ideal~\cite{khodam}.
\subsection*{Background on uniform spaces} For collection $\mathcal F$ of subsets of $X\times X$, we say $\mathcal F$
is a uniform structure on $X$ if (let $\Delta_X:=\{(x,x):x\in X\}$):
\begin{itemize}
\item $\forall\mathcal{O}\in\mathcal{F}\:\:\Delta_X\subseteq\mathcal{O}$,
\item $\forall\mathcal{O},\mathcal{U}\in\mathcal{F}\:\:\mathcal{O}\cap\mathcal{U}\in\mathcal{F}$,
\item $\forall\mathcal{O}\in\mathcal{F}\:\:\forall\mathcal{U}\subseteq X\times X\:\:(\mathcal{O}\subseteq\mathcal{U}\Rightarrow\mathcal{U}\in\mathcal{F})$,
\item $\forall\mathcal{O}\in\mathcal{F}\:\:{\mathcal O}^{-1}\in \mathcal{F}$,
\item $\forall\mathcal{O}\in\mathcal{F}\:\:\exists\mathcal{U}\in\mathcal{F}\:\:\mathcal{U}\circ\mathcal{U}\subseteq \mathcal{O}$.
\end{itemize}
Then $\{D\subseteq X:\forall x\in
X\:\:\exists\mathcal{O}\in\mathcal{F}\:\:\mathcal{O}[x]\subseteq
D\}$ where $\mathcal{O}[x]:=\{y\in X:(x,y)\in\mathcal{O}\}$ (for
$\mathcal{O}\in\mathcal{F}$ and $x\in X$) is a topology on $X$
called uniform topology on $X$ induced by $\mathcal F$. We say
topological space $Y$ is uniformizable if there exists a uniform
structure $\mathcal{H}$ on $Y$ such that uniform topology induced
by $\mathcal H$ on $Y$ is compatible with original topology on
$Y$, moreover in the above case we say $\mathcal H$ is a
compatible uniform structure on $Y$. Let's recall that compact
Hausdorff topological space $Z$ has a unique compatible uniform
structure $\{\mathcal{O}\subseteq Z\times Z:\mathcal{O}$ contains
an open neighbourhood of $\Delta_Z\}$. For details on uniform
spaces see \cite{dug}.
\subsection*{Background on (weighted) generalized shifts}
\noindent One--sided and two--sided shifts are amongst most
useful tools in different areas including ergodic theory and
dynamical systems which have been studied and developed earlier
than 60 years ago, one may find their first motivations in
Kolomogorov, Sinai, Ornstein and other mathematicains'
works~\cite{casa, ornstein, walters}. One may find different
generalizations of the above shifts regarding their point of
view, e.g. as it has been mentioned in~2006~in~\cite{Marchenko}
``\textit{Between 1938 and 1940 Jean Delsarte and B. M. Levitan
developed the theory of generalized shift operators $T_x^y[f]$,
which map functions $f(x)$ into functions of two variables
$T_x^y[f]$ and satisfy four axioms that generalized the
properties of ordinary shift. Among these axioms ...}''. However
in many works, special case of the above definition, i.e., $p$
times iterated two--sided shift have been considered as
generalized shifts too~\cite{calcul}. In the following text we
will use none of the above two generalizations, we use point of
view, notation and definition of generalized shift established and
introduced for the first time in~\cite{note} (again as a
generalization of one--sided  and two--sided shift) which has
been appeared in dynamical and non--dynamical papers
(e.g.~\cite{abad, anna-b}): let's recall that for arbitrary
nonempty sets $X,\Gamma$ and $\varphi:\Gamma\to\Gamma$, we call
$\sigma_\varphi:X^\Gamma\to X^\Gamma$ with
$\sigma_\varphi((x_\alpha)_{\alpha\in\Gamma})=(x_{\varphi(\alpha)}
)_{\alpha\in\Gamma}$ for $(x_\alpha)_{\alpha\in\Gamma}\in
X^\Gamma$, a generalized shift~\cite{note}.
\begin{definition}
Suppose $M$ is a module over ring $R$, $\Gamma$ is a nonempty set,
$\mathfrak{w}=({\mathfrak w}_\alpha)_{\alpha\in\Gamma}\in R^\Gamma$ is a ``weight vector'' and
$\varphi:\Gamma\to\Gamma$ is arbitrary, we call $\sigma_{\varphi,\mathfrak{w}}:M^\Gamma\to M^\Gamma$
with $\sigma_{\varphi,\mathfrak{w}}((x_\alpha)_{\alpha\in\Gamma})=(\mathfrak{w}_\alpha x_{\varphi(\alpha)})_{\alpha
\in\Gamma}$ for $(x_\alpha)_{\alpha\in\Gamma}\in M^\Gamma$ a weighted generalized shift. Note that for topological
module $M$, weighted generalized shift $\sigma_{\varphi,\mathfrak{w}}:M^\Gamma\to M^\Gamma$ is continuous, where
$M^\Gamma$ has been equipped with product topology.
\end{definition}
\noindent Weighted generalized shifts in some point of view are
just weighted composition operators~\cite{comp} which are main
interest of many mathematicians, however in the above point of
view, they have  been introduced for the first time
in~\cite{fweight} as a common generalization of ``generalized
shifts'' and ``weighted shifts'',  also study over this concept
has been continued in~\cite{weight}.
\begin{convention}
In the following text consider discrete finite abelian ring $R$ with unity $1$ and zero element $0$,
arbitrary set $\Gamma$ with at least two elements, self--map $\varphi:\Gamma\to\Gamma$ and weight vector
$\mathfrak{w}=({\mathfrak w}_\alpha)_{\alpha\in\Gamma}\in R^\Gamma$. Moreover equip
$R^\Gamma$ with product topology and
(unique) compatible uniform structure
\begin{center}
$\mathcal{K}=\{\mathcal{O}\subseteq R^\Gamma\times R^\Gamma:$
there exists finite subset $M$ of $\Gamma$ with $\mathcal{O}_M\subseteq\mathcal{O}\}$
\end{center}
where for $M\subseteq \Gamma$, we have
\[\mathcal{O}_M:=\{((x_\alpha)_{\alpha\in\Gamma},(y_\alpha)_{\alpha\in\Gamma})
\in R^\Gamma\times R^\Gamma:\forall\alpha\in
M\:\:(x_\alpha=y_\alpha)\}\:. \] For $\theta\in\Gamma$ and $r\in
R$, let $U(\theta,r)=\mathop{\prod}_{\alpha\in\Gamma}U_\alpha$
where $U_\theta=\{r\}$ and $U_\alpha=R$ for $\alpha\neq\theta$.
So $\{U(\theta,r):\theta\in\Gamma,r\in R\}$ is a sub--base of
product topology on $R^\Gamma$, note that for $r_1,\ldots,r_n\in
R$ and distinct $\alpha_1,\ldots,\alpha_n\in\Gamma$, we have
$\mathop{\bigcap}\limits_{1\leq i\leq
n}U(\alpha_i,r_i)=\mathop{\prod}\limits_{\alpha\in\Gamma}V_\alpha$
for $V_{\alpha_i}=\{r_i\}$ ($1\leq i\leq n$) and $V_\alpha=R$ for
$\alpha\neq\alpha_1,\ldots,\alpha_n$.
\end{convention}
\section{Sensitivity in $(R^\Gamma,\sigma_{\varphi,{\mathfrak w}})$}
\noindent In metric space $(X,d)$ for $\varepsilon>0$, let
$\alpha_\varepsilon:=\{(x,y)\in X\times X:d(x,y)<\varepsilon\}$,
then $\mathcal{F}_d:=\{\mathcal{O}\subseteq X\times
X:\exists\varepsilon>0\:\:\alpha_\varepsilon\subseteq\mathcal{O}\}$
is a compatible uniform structure on $X$ and it is easy to see
that $(X,f)$ is sensitive if and only if there exists
$\mathcal{O}\in {\mathcal F}_d$ such that for each $x\in X$ and
open neighbourhood $U$ of $x$, there exists $y\in U$ and $n\geq0$
such that $(f^n(x),f^n(y))\notin\mathcal{O}$. Note that
sensitivity of $(X,f)$ depends on chosen compatible metric $d$ on
$X$ \cite{fed}. This comparison leads us to the following
definition:
\begin{definition}
In uniform space $(X,\mathcal F)$ we say dynamical system
$(X,f)$ is
\begin{itemize}
\item sensitive if there exists entourage $\mathcal{O}\in\mathcal{F}$, such that for all $x\in X$ and all open
    neighbourhood $U$ of $x$, there exists $y\in U$ and $n\geq0$ with
    \linebreak $(f^n(x),f^n(y))\notin\mathcal{O}$~\cite[Definition 3]{wu} (see~\cite[Definition 4.2]{jaleb} too),
\item strongly sensitive if there exists entourage $\mathcal{O}\in\mathcal{F}$, such that for all $x\in X$ and all open
    neighbourhood $U$ of $x$, there exists $y\in U$ and $N\geq0$ with $(f^n(x),f^n(y))\notin\mathcal{O}$ for all $n\geq N$
    (see~\cite[Definition 3.1]{abraham} too).
\end{itemize}
\end{definition}
\noindent Note that sensitivity of $(X,f)$ depends on chosen
compatible uniformity of $X$. However since every compact
Hausdorff space has a unique compatible uniform structure, so in
this case no need to specify compatible uniform structure on
compact Hausdorff space $X$.
\\
In this section we prove that for finite field $R$,
$(R^\Gamma,\sigma_{\varphi,\mathfrak{w}})$ is (strongly) sensitive if and only if
there exists non--quasi--periodic point $\theta\in\Gamma$ such that for all $n\geq0$,
$\mathfrak{w}_{\varphi^n(\theta)}\neq0$.
\begin{remark}
For self--map $h:A\to A$, we say $a\in A$ is a:
\\ $\bullet$ periodic point of $h$, if there exists $n\geq1$ such that $h^n(a)=a$,
\\ $\bullet$ quasi--periodic of $h$, if there exist $i>j\geq1$ such that $h^i(a)=h^j(a)$ (or equivalently
$\{h^n(a)\}_{n\geq1}$ is finite) (known as: quasi--periodic point~\cite[Definition 2.1]{abad}, eventually periodic point~\cite{nathan}, pre--periodic point~\cite{ben} too),
\\ $\bullet$ non--quasi--periodic of $h$, if it is not a quasi--periodic point of $h$, or equivalently
$\{h^n(a)\}_{n\geq1}$ is infinite (known as wandering point too~\cite[Definition 2.1]{abad}).
\end{remark}
\begin{lemma}\label{salam-sen-10}
If for all $\alpha\in\Gamma$:
\\
$\bullet$ either $\alpha$ is a quasi--periodic point of $\varphi$,
\\
$\bullet$ or $\alpha$ is a non--quasi--periodic point of $\varphi$ and there exists $n\geq0$ with ${\mathfrak w}_{\varphi^n(\alpha)}=0$,
\\
then $(R^\Gamma,\sigma_{\varphi,\mathfrak{w}})$ is not sensitive.
\end{lemma}
\begin{proof}
Suppose for each $\alpha\in\Gamma$ either $\alpha$ is a quasi--periodic point of $\varphi$ or there exists $n\geq0$,
with ${\mathfrak w}_{\varphi^n(\alpha)}=0$.
Consider $\mathcal{O}\in\mathcal{K}$, there exists finite subset $M$ of $\Gamma$ with $\mathcal{O}_M\subseteq\mathcal{O}$.  \begin{center}
\begin{tabular}{cc}
$\Lambda:=\{\varphi^n(\alpha):n\geq0$ and $\alpha\in M$ is a quasi--periodic point of $\varphi\}\cup$ &
\\
$\{\varphi^n(\alpha):\alpha\in M$ is a non--quasi--periodic point of $\varphi$  & ($\divideontimes$)
\\
and $0\leq n\leq\min\{k\geq0:
\mathfrak{w}_{\varphi^k(\alpha)}=0\}\}$
&
 \\ \end{tabular}
\end{center}
then $\Lambda$ is a finite subset of $\Gamma$. Consider $(x_\alpha)_{\alpha\in\Gamma}\in R^\Gamma$, then
\[U=\{(y_\alpha)_{\alpha\in\Gamma}\in R^\Gamma:\forall\alpha\in \Lambda, \:y_\alpha=x_\alpha\}(=
\mathop{\bigcap}\limits_{\alpha\in \Lambda}U(\alpha,x_\alpha))\]
is an open neighbourhood of
$(x_\alpha)_{\alpha\in\Gamma}$. Consider $(y_\alpha)_{\alpha\in\Gamma}\in U$, for $\lambda\in\Lambda$
and $n\geq0$ we have the following cases:
\begin{itemize}
\item[a.] $\lambda\in \Lambda$ is a quasi--periodic point of $\varphi$. In this case
    $\{\varphi^i(\lambda):i\geq0\}\subseteq\Lambda$, thus $\varphi^n(\lambda)\in\Lambda$
    and $x_{\varphi^n(\lambda)}=y_{\varphi^n(\lambda)}$ which shows
    $\mathfrak{w}_\lambda\mathfrak{w}_{\varphi(\lambda)}\cdots\mathfrak{w}_{\varphi^{n-1}(\lambda)}x_{\varphi^n(\lambda)}
    =\mathfrak{w}_\lambda\mathfrak{w}_{\varphi(\lambda)}\cdots\mathfrak{w}_{\varphi^{n-1}(\lambda)}y_{\varphi^n(\lambda)}$.
\item[b.] $\lambda\in\Lambda$ is a non--quasi--periodic point of $\varphi$ and
    $0\leq n\leq\min\{k\geq0:\mathfrak{w}_{\varphi^k    (\lambda)}=0\}$.
    In this case $\varphi^n(\lambda)\in\Lambda$ and using a similar method described in (a) we have
    $\mathfrak{w}_\lambda\mathfrak{w}_{\varphi(\lambda)}\cdots\mathfrak{w}_{\varphi^{n-1}(\lambda)}x_{\varphi^n(\lambda)}
    =\mathfrak{w}_\lambda\mathfrak{w}_{\varphi(\lambda)}\cdots\mathfrak{w}_{\varphi^{n-1}(\lambda)}y_{\varphi^n(\lambda)}$.
\item[c.] $\lambda\in\Lambda$ is a non--quasi--periodic point of $\varphi$ and $n>\min\{k\geq0:\mathfrak{w}_{\varphi^k
    (\lambda)}=0\}$. In this case
    $\mathfrak{w}_\lambda\mathfrak{w}_{\varphi(\lambda)}\cdots\mathfrak{w}_{\varphi^{n-1}(\lambda)}x_{\varphi^n(\lambda)}
    =0=\mathfrak{w}_\lambda\mathfrak{w}_{\varphi(\lambda)}\cdots\mathfrak{w}_{\varphi^{n-1}(\lambda)}y_{\varphi^n(\lambda)}$.
\end{itemize}
Using the above cases:
\[\forall\lambda\in\Lambda\:\:\forall n\geq 0,\:\:(\mathfrak{w}_\lambda\mathfrak{w}_{\varphi(\lambda)}\cdots\mathfrak{w}_{\varphi^{n-1}(\lambda)}x_{\varphi^n(\lambda)}
    =\mathfrak{w}_\lambda\mathfrak{w}_{\varphi(\lambda)}\cdots\mathfrak{w}_{\varphi^{n-1}(\lambda)}y_{\varphi^n(\lambda)})\]
i.e., for all $n\geq0$ we have:
\\
$(\sigma^n_{\varphi,\mathfrak{w}}((x_\alpha)_{\alpha\in\Gamma}) , \sigma^n_{\varphi,\mathfrak{w}}((y_\alpha)_{\alpha\in\Gamma})) $
\begin{eqnarray*}
& = & ((\mathfrak{w}_\alpha\mathfrak{w}_{\varphi(\alpha)}\cdots\mathfrak{w}_{\varphi^{n-1}(\alpha)}x_{\varphi^n(\alpha)})_{\alpha\in\Gamma}, (\mathfrak{w}_\alpha\mathfrak{w}_{\varphi(\alpha)}\cdots\mathfrak{w}_{\varphi^{n-1}(\alpha)}y_{\varphi^n(\alpha)})_{\alpha\in\Gamma})\\
& \in & \mathcal{O}_\Lambda\subseteq\mathcal{O}_M\subseteq\mathcal{O}
\end{eqnarray*}
Therefore for all $\mathcal{O}\in\mathcal{K}$ and $x\in
R^\Gamma$, there exists open neighbourhood $U$ of $x$ such that
$(\sigma^n_{\varphi,\mathfrak{w}}(x) ,
\sigma^n_{\varphi,\mathfrak{w}}(y))\in\mathcal{O}$ and
$(R^\Gamma,\sigma_{\varphi,\mathfrak{w}})$ is not sensitive.
\end{proof}
\noindent Let's recall that we say $x\in R\setminus\{0\}$ is invertible if there exists $y\in R\setminus\{0\}$
such that $xy=yx=1$.
\begin{lemma}\label{salam-sen-20}
If there exists non--quasi--periodic point $\theta\in\Gamma$ such that for all $n\geq0$,
$\mathfrak{w}_{\varphi^n(\theta)}$ is invertible, then
$(R^\Gamma,\sigma_{\varphi,\mathfrak{w}})$ is (strongly) sensitive.
\end{lemma}
\begin{proof}
Suppose $\theta$ is a non--quasi--periodic point of $\varphi$ (i.e., $\{\varphi^n(\theta)\}_{n\geq0}$ is a
one--to--one sequence) such that for all $n\geq0$, $\mathfrak{w}_{\varphi^n(\theta)}$ is invertible. Consider
$x=(x_\alpha)_{\alpha\in\Gamma}\in R^\Gamma$ and open neighbourhood $U$ of $x$, there exists finite subset
$M$ of $\Gamma$ such that $x\in\mathop{\bigcap}\limits_{\alpha\in M}U(\alpha,x_\alpha)\subseteq U$.
Since $\{\varphi^n(\theta)\}_{n\geq0}$ is a
one--to--one sequence, there exists $N\geq1$ such that
\[\forall n\geq N\:\:(\varphi^n(\theta)\notin M)\:,\]
in particular
{\small \[\{(y_\alpha)_{\alpha\in\Gamma}\in R^\Gamma:\forall \alpha\notin\{\varphi^N(\theta),\varphi^{N+1}(\theta),\ldots\}\:\:(x_\alpha=y_\alpha)\}\subseteq
\mathop{\bigcap}\limits_{\alpha\in M}U(\alpha,x_\alpha)\subseteq U\:.\tag{$\divideontimes\divideontimes$}\]}
Choose $z=(z_\alpha)_{\alpha\in\Gamma}\in R^\Gamma$ such that:
\[z_\alpha=\left\{\begin{array}{lc} an\: element\: of\: R\setminus\{x_{\varphi^m(\theta)}\} & m\geq N,\alpha=\varphi^m(\theta)\:, \\
x_\alpha &
\alpha\notin\{\varphi^N(\theta),\varphi^{N+1}(\theta),\varphi^{N+2}(\theta),\ldots\}\:,
\end{array}\right.\] then by ($\divideontimes\divideontimes$),
$z\in U$. Moreover for all $m\geq N$ we have:
\vspace{5mm} \\
$z_{\varphi^m(\theta)}\in R\setminus\{x_{\varphi^m(\theta)}\}$
\begin{eqnarray*}
 & \Rightarrow & z_{\varphi^m(\theta)}\neq x_{\varphi^m(\theta)} \\
& \Rightarrow &
    \mathfrak{w}_\theta\mathfrak{w}_{\varphi(\theta)}\cdots\mathfrak{w}_{\varphi^{m-1}(\theta)} z_{\varphi^m(\theta)}
    \neq \mathfrak{w}_\theta\mathfrak{w}_{\varphi(\theta)}\cdots\mathfrak{w}_{\varphi^{m-1}(\theta)} x_{\varphi^m(\theta)} \\
&& \SP\SP\SP\SP\SP\SP\SP\SP\SP\SP\SP\SP\SP\SP\SP
    (since\:{\mathfrak w}_{\varphi^i(\theta)}s\: are \: invertible) \\
& \Rightarrow & ((\mathfrak{w}_\alpha\mathfrak{w}_{\varphi(\alpha)}\cdots\mathfrak{w}_{\varphi^{m-1}(\alpha)}
    z_{\varphi^m(\alpha)})_{\alpha\in\Gamma},(\mathfrak{w}_\alpha\mathfrak{w}_{\varphi(\alpha)}\cdots
    \mathfrak{w}_{\varphi^{m-1}(\alpha)}x_{\varphi^m(\alpha)})_{\alpha\in\Gamma})\notin\mathcal{O}_{\{\theta\}} \\
& \Rightarrow & (\sigma^m_{\varphi,\mathfrak{w}}(z),\sigma^m_{\varphi,\mathfrak{w}}(x))\notin\mathcal{O}_{\{\theta\}}
\end{eqnarray*}
Hence for all $x\in R^\Gamma$ and open neighborhood $U$ of $x$,
there exists $z\in U$ and $N\geq1$ such that for all $m\geq N$,
$(\sigma^m_{\varphi,\mathfrak{w}}(z),\sigma^m_{\varphi,\mathfrak{w}}(x))\notin\mathcal{O}_{\{\theta\}}
(\in\mathcal{K})$, and $(R^\Gamma,\sigma_{\varphi,\mathfrak{w}})$
is (strongly) sensitive.
\end{proof}
\begin{corollary}\label{salam-sen-30}
In finite field $R$ the following statements are equivalent:
\begin{itemize}
\item[a.] $(R^\Gamma,\sigma_{\varphi,\mathfrak{w}})$ is sensitive,
\item[b.] $(R^\Gamma,\sigma_{\varphi,\mathfrak{w}})$ is strongly sensitive,
\item[c.] there exists non--quasi--periodic point $\theta\in\Gamma$ such that for all $n\geq0$,
$\mathfrak{w}_{\varphi^n(\theta)}\neq0$.
\end{itemize}
\end{corollary}
\begin{proof}
Use Lemmas~\ref{salam-sen-10} and \ref{salam-sen-20}, and the fact that all nonzero elements of $R$ are invertible.
\end{proof}
\noindent The following counterexample shows that if we omit the assumption of being $R$ a finite field, then
Corollary~\ref{salam-sen-30} may fails to be true.
\begin{counterexample}
For $\mathbb{Z}_4=\{\overline{0},\overline{1},\overline{2},\overline{3}\}(=\frac{\mathbb Z}{4\mathbb{Z}})$,
$\mathfrak{v}=(\mathfrak{v}_n)_{n\in\mathbb{N}}:=(\overline{2})_{n\in\mathbb{N}}$, and
$\psi:\mathop{\mathbb{N}\to\mathbb{N}}\limits_{n\mapsto n+1}$,
then for all $n\geq0$ we have $\mathfrak{v}_{\psi^n(1)}=\mathfrak{v}_{n+1}=\overline{2}\neq\overline{0}$.
Hence $(\mathbb{Z}_4^{\mathbb N},\sigma_{\psi,\mathfrak{v}})$ satisfies item (c) in Corollary~\ref{salam-sen-30}.
However $(\mathbb{Z}_4^{\mathbb N},\sigma_{\psi,\mathfrak{v}})$ is not strongly sensitive since for all
$x,y\in \mathbb{Z}_4^{\mathbb N}$ and $k\geq2$ we have $\sigma_{\psi,\mathfrak{v}}^k(x)=\sigma_{\psi,\mathfrak{v}}^k(y)=(\overline{0})_{n\in\mathbb{N}}$.
\end{counterexample}
\section{Li--Yorke chaotic $(R^\Gamma,\sigma_{\varphi,{\mathfrak w}})$, for finite field $R$}
\noindent Let's recall that by transformation semigroup
$(S,X,\rho)$ we mean a discrete topological semigroup $S$ with
identity $e$, topological space $X$ and continuous map
$\rho:S\times X\to X$ with $\rho(s,x)=:sx$ such that for all
$x\in X$ and $s,t\in S$ we have $ex=x$ and $(st)x=s(tx)$. It is
well known that the collection of all dynamical systems and  the
collection of all transformation semigroups with phase semigroup
$\mathbb{N}\cup\{0\}$ are in one--to--one correspondence in the
following sense:
A dynamical system $(X,f)$ is just the transformation semigroup
$(\mathbb{N}\cup\{0\},X,\rho_f)$ where $\rho_f(n,x):=f^n(x)$ for
all $n\geq0$ and $x\in X$.
\\ So it is possible to adopt the definition of Li--Yorke chaos from transformation semigroups with uniform phase space to dynamical systems with uniform phase spaces, i.e. whenever we say in dynamical system $(X,f)$ two points
$x,y\in X$ are proximal (resp. asymptotic, scrambled, ...) we mean $x,y$ are proximal (resp. asymptotic, scrambled, ...)
in transformation semigroup $(\mathbb{N}\cup\{0\},X,\rho_f)$, moreover whenever we say the dynamical system $(X,f)$
is Li--Yorke chaotic, it means the transformation semigroup $(\mathbb{N}\cup\{0\},X,\rho_f)$ is Li--Yorke chaotic.
\\
In dynamical system $(X,f)$ and transformation semigroup
$(S,X,\rho)$ with compact Hausdorff phase space $X$ and unique
compatible uniform structure $\mathcal F$, we have the following
definitions~\cite[Definitions 2.1 and 2.2]{khodam}:
\\
{\bf Proximal pair and proximal relation.} For $x,y\in X$, we say
$x,y$ are proximal in transformation semigroup $(S,X,\rho)$, if
there exists a net $\{s_{\lambda}\}_{\lambda\in\Lambda}$ in $S$
    and $z\in X$ with $\mathop{\lim}\limits_{\lambda\in\Lambda}s_\lambda x=z=
    \mathop{\lim}\limits_{\lambda\in\Lambda}s_\lambda y$.
In $(X,f)$ the following statements are equivalent:
\begin{itemize}
\item  $x,y$ are proximal in dynamical system $(X,f)$
\item for each $\mathcal{O}\in{\mathcal F}$,
    $\{n\geq1:(f^n(x),f^n(y))\in\mathcal{O}\}$ is infinite
\item for each $\mathcal{O}\in{\mathcal F}$,
    $\{n\geq1:(f^n(x),f^n(y))\in\mathcal{O}\}$ is nonempty.
\end{itemize}
We denote the collection of all proximal pairs of dynamical system $(X,f)$ by $Prox(X,f)$ or simply $Prox (f)$.
\\
{\bf Asymptotic pair and asymptotic relation.}  For $x,y\in X$,
we say $x,y$ are asymptotic     modul $\mathcal{P}_{fin}(S)(=\{A\subseteq S:A$ is finite$\}$) in $(S,X,\rho)$, if for each
$\mathcal{O}\in{\mathcal F}$,
    $\{s\in S:(sx,sy)\notin\mathcal{O}\}$ is finite. We denote the collection of all asymptotic pairs of dynamical system $(X,f)$
    by $Asym(X,f)$ or simply $Asym(f)$. Hence
    \[Asym(f)=\{(x,y)\in X\times X:\forall\mathcal{O}\in\mathcal{F}\:\exists N
    \geq1\:\forall n\geq N\:(f^n(x),f^n(y))\in\mathcal{O}\}\]
    in particular, $Asym(f)\subseteq Prox(f)$.
\\
{\bf Scrambled pair and scrambled subset.} We say $x,y\in X$ are
scrambled if they are proximal and they are not asymptotic.
$A\subseteq X$ with at least two elements is scrambled if all
distinct elements $z,w\in A$ are scrambled.
\\
{\bf Li--Yorke chaotic.} $(X,f)$ is Li--Yorke chaotic if it has an uncountable scrambled subset.
\\
In this section we prove that for finite field $R$ the following statements are equivalent:
\begin{itemize}
\item $(R^\Gamma,\sigma_{\varphi,\mathfrak{w}})$ is sensitive,
\item $(R^\Gamma,\sigma_{\varphi,\mathfrak{w}})$ has at least one scrambled pair,
\item $(R^\Gamma,\sigma_{\varphi,\mathfrak{w}})$ is Li--Yorke chaotic.
\end{itemize}
\begin{note}
For $x,y\in R^\Gamma$ the following statements are equivalent:
\begin{itemize}
\item[1.] $(x,y)\in Asym(\sigma_{\varphi,\mathfrak{w}})$
\item[2.] for all ${\mathcal O}\in{\mathcal K}$, $\{n\geq1:(\sigma_{\varphi,\mathfrak{w}}^n(x),\sigma_{\varphi,\mathfrak{w}}^n(y))\notin{\mathcal O}\}$ is finite,
\item[3.] for all finite subset $M$ of $\Gamma$, $\{n\geq1:(\sigma_{\varphi,\mathfrak{w}}^n(x),\sigma_{\varphi,\mathfrak{w}}^n(y))\notin{\mathcal O}_M\}$ is finite,
\item[4.] for all $\theta\in\Gamma$, $\{n\geq1:(\sigma_{\varphi,\mathfrak{w}}^n(x),\sigma_{\varphi,\mathfrak{w}}^n(y))\notin{\mathcal O}_{\{\theta\}}\}$ is finite.
\end{itemize}
Since for all ${\mathcal O}\in{\mathcal K}$ there exists finite subset $M$ of $\Gamma$ with
${\mathcal O}_M\subseteq\mathcal{O}$, also for all $\theta_1,\ldots,\theta_n\in\Gamma$ we have
${\mathcal O}_{\{\theta_1,\ldots,\theta_n\}}=\bigcap\{{\mathcal O}_{\{\theta_i\}}:1\leq i\leq n\}$.
\end{note}
\begin{lemma}\label{salam-yorke-10}
Suppose for all $\alpha\in\Gamma$:
\\
$\bullet$ either $\alpha$ is a quasi--periodic point of $\varphi$,
\\
$\bullet$ or $\alpha$ is a non--quasi--periodic point of $\varphi$ and there exists $n\geq0$ with
${\mathfrak w}_{\varphi^n(\alpha)}=0$,
\\
then $Prox(\sigma_{\varphi,\mathfrak{w}})= Asym(\sigma_{\varphi,\mathfrak{w}})$.
\end{lemma}
\begin{proof}
Suppose for all $\alpha\in\Gamma$, $\alpha$ is a quasi--periodic
point of $\varphi$ or there exists $n\geq0$ with ${\mathfrak
w}_{\varphi^n(\alpha)}=0$.
We prove $Prox(\sigma_{\varphi,\mathfrak{w}})\subseteq Asym(\sigma_{\varphi,\mathfrak{w}})$.
Consider
$(x,y)=((x_\alpha)_{\alpha\in\Gamma},(y_\alpha)_{\alpha\in\Gamma})\in
Prox(\sigma_{\varphi,\mathfrak{w}}) \setminus
Asym(\sigma_{\varphi,\mathfrak{w}})$. Since $(x,y)\notin
Asym(\sigma_{\varphi,\mathfrak{w}})$, there exists
$\theta\in\Gamma$ such that
\[T_1:=\{n\geq1:(\sigma_{\varphi,\mathfrak{w}}^n(x),\sigma_{\varphi,\mathfrak{w}}^n(y))\notin{\mathcal O}_{\{\theta\}}\}\]
is infinite.
Consider $\Lambda$ as
($\divideontimes$) in the proof of Lemma~\ref{salam-sen-10} for $M=\{\theta\}$, hence:
\[\Lambda:=\left\{\begin{array}{lc} \{\varphi^i(\theta):i\geq0\}\:, & \theta{\rm \: is \: a \: quasi-periodic \: point \: of \: }\varphi \:,\\
\{\theta,\varphi(\theta),\ldots,\varphi^i(\theta)\}\:, & \theta{\rm \: is \: a \: non-quasi-periodic \: point \: of \: }  \\
& \varphi {\rm \: and\:}
i=\min\{n\geq0:\mathfrak{w}_{\varphi^n(\theta)}=0\}\:,
\end{array}\right.\]
then $\Lambda$ is a finite subset of $\Gamma$. Choose $m>\:|\Lambda|+1$ with $m\in T_1$. By
$(\sigma_{\varphi,\mathfrak{w}}^m(x),
    \sigma_{\varphi,\mathfrak{w}}^m(y))\notin{\mathcal O}_{\{\theta\}}$ we have
    $\mathfrak{w}_\theta\mathfrak{w}_{\varphi(\theta)}\cdots\mathfrak{w}_{\varphi^{m-1}(\theta)}
    x_{\varphi^m(\theta)}\neq\mathfrak{w}_\theta\mathfrak{w}_{\varphi(\theta)}\cdots\mathfrak{w}_{\varphi^{m-1}(\theta)}
    y_{\varphi^m(\theta)}$ which shows $x_{\varphi^m(\theta)}\neq y_{\varphi^m(\theta)}$
and $\mathfrak{w}_{\varphi^i(\theta)}\neq0$ for all $i\in\{0,\ldots,|\Lambda|+1\}$. Hence not only $\theta$ is a quasi--periodic point
of $\varphi$ but also $\mathfrak{w}_{\varphi^j(\theta)}\neq0$ for all $j\geq0$.
\\
Since $(x,y)\in Prox(\sigma_{\varphi,\mathfrak{w}})$, there exists $l\geq1$ such that
$(\sigma_{\varphi,\mathfrak{w}}^l(x),\sigma_{\varphi,\mathfrak{w}}^l(y))\in{\mathcal O}_\Lambda$, thus
$\mathfrak{w}_\alpha\mathfrak{w}_{\varphi(\alpha)}\cdots\mathfrak{w}_{\varphi^{l-1}(\alpha)}x_{\varphi^l(\alpha)}=
    \mathfrak{w}_\alpha\mathfrak{w}_{\varphi(\alpha)}\cdots\mathfrak{w}_{\varphi^{l-1}(\alpha)}y_{\varphi^l(\alpha)}$
    for all $\alpha\in\Lambda$, i.e.$ \mathfrak{w}_{\varphi^n(\theta)}\cdots
    \mathfrak{w}_{\varphi^{n+l-1}(\theta)}x_{\varphi^{n+l}(\theta)}=
        \mathfrak{w}_{\varphi^n(\theta)}\cdots
    \mathfrak{w}_{\varphi^{n+l-1}(\theta)}y_{\varphi^{n+l}(\theta)}$ for all $n\geq0$. Hence
\[\forall n\geq l\:\:
x_{\varphi^{n}(\theta)}=y_{\varphi^{n}(\theta)} \]
so
$(\sigma_{\varphi,\mathfrak{w}}^n(x),
    \sigma_{\varphi,\mathfrak{w}}^n(y))\in{\mathcal O}_{\{\theta\}}
$ for all $n\geq l$ which is in contradiction with
infiniteness of $T_1$.
\end{proof}
\begin{remark}\label{salam-yorke-20}
There exists an uncountable collection $\mathcal E$ of infinite subsets of $\mathbb N$ such that for each distinct
$E,F\in\mathcal{E}$, the set $E\cap F$ is finite~\cite{large}.
\end{remark}
\begin{definition}
For $h:A\to A$ define equivalence relation $\thicksim_h$ on $A$ with
\[x\thicksim_h y\Leftrightarrow (\exists n,m\geq0 \:\: h^n(x)=h^m(y))\]
for all $x,y\in A$. If $h:A\to A$ is one--to--one, then for every
equivalence class $D\in \frac{A}{\thicksim_h}$ exactly one of the
following conditions occurs:
\begin{itemize}
\item $D$ is finite and for all $\alpha\in D$, we have $D=\{h^n(\alpha):n\geq0\}\subseteq Per(h)$,
\item $D$ is infinite and there exists unique $\alpha\in D$ such that
    $D=\{h^n(\alpha):n\geq0\}$  (so $\{h^n(\alpha)\}_{n\geq0}$ is
    a one--to--one sequence),
\item $D$ is infinite and for all $\alpha\in D$ and $n\in{\mathbb Z}$
    we have $h^n(\alpha)\neq\varnothing$, moreover
    $D=\{h^n(\alpha):n\in{\mathbb Z}\}$  (so $\{h^n(\alpha)\}_{n\in{\mathbb Z}}$ is
    a one--to--one bi--sequence),
\end{itemize}
\end{definition}
\subsection{Lemmas on a special case}\label{ABC}
In the following string of Lemmas, i.e., \ref{ABC10}, \ref{ABC20}, \ref{ABC30}
suppose $\nu\in\Gamma$ is a non--quasi--periodic point of $\varphi$
such that $\mathfrak{w}_{\varphi^n(\nu)}\neq0$, for all $n\geq0$, if $A\subseteq\mathbb{N}\cup\{0\}$ let:
\[\xi_A(n):=\left\{\begin{array}{lc} 1 & n\in A\:, \\ 0 & n\notin A\:,\end{array}\right.\]
and $x^{A}:=(x^{A}_\alpha)_{\alpha\in\Gamma}$ where:
\[x_\alpha^A=\left\{\begin{array}{lc} \xi_A(p) & \alpha=\varphi^{\frac{p(p+1)}{2}}(\nu),p\geq0 \:, \\
0 & otherwise\:. \end{array}\right.\]
\begin{lemma}\label{ABC10}
If $E,F\subseteq \mathbb{N}\cup\{0\}$ with infinite $E\setminus F$,
then $(x^E,x^F)\notin Asym(\sigma_{\varphi,\mathfrak{w}})$.
\end{lemma}
\begin{proof}
For $r\in E\setminus F$, we have
\[
x^E_{\varphi^{\frac{r(r+1)}2}(\nu)}=\xi_E(r)=1\:\:,\:\:
x^F_{\varphi^{\frac{r(r+1)}2}(\nu)}=\xi_F(r)=0\:, \]
which leads to:
\[\begin{array}{c} \mathfrak{w}_\nu\mathfrak{w}_{\varphi(\nu)}\cdots\mathfrak{w}_{\varphi^{\frac{r(r+1)}2-1}(\nu)}x^{E}_{\varphi^{\frac{r(r+1)}2}(\nu)}\neq 0 \:,\\ \\
\mathfrak{w}_\nu\mathfrak{w}_{\varphi(\nu)}\cdots\mathfrak{w}_{\varphi^{\frac{r(r+1)}2-1}(\nu)}x^{F}_{\varphi^{\frac{r(r+1)}2}(\nu)}= 0 \:,
\end{array}\]
and $(\sigma_{\varphi,\mathfrak{w}}^{r}((x_\alpha^E)_{\alpha\in\Gamma}),\sigma_{\varphi,\mathfrak{w}}^{r}((x_\alpha^F)_{\alpha\in\Gamma}))\notin\mathcal{O}_{\{\nu\}}$. Hence:
\[E\setminus F\subseteq\{n\geq1:(\sigma_{\varphi,\mathfrak{w}}^{n}((x_\alpha^E)_{\alpha\in\Gamma}),\sigma_{\varphi,\mathfrak{w}}^{n}((x_\alpha^F)_{\alpha\in\Gamma}))\notin\mathcal{O}_{\{\nu\}}\}\]
and $\{n\geq1:(\sigma_{\varphi,\mathfrak{w}}^{n}(x^E),\sigma_{\varphi,\mathfrak{w}}^{n}(x^F))\notin\mathcal{O}_{\{\nu\}}\}$ is infinite, therefore:
\[(x^{E},x^{F})\notin Asym(\sigma_{\varphi,\mathfrak{w}})\:.\]
\end{proof}
\begin{lemma}\label{ABC20}
Suppose $M$ is a finite subset
of $\frac{\nu}{\thicksim_\varphi}$ and $E,F\subseteq{\mathbb N}\cup\{0\}$, then
\[T_M:=\{n\geq1:(\sigma_{\varphi,\mathfrak{w}}^n(x^E),\sigma_{\varphi,\mathfrak{w}}^n(x^F))\in \mathcal{O}_M\}\]
is infinite.
\end{lemma}
\begin{proof}
Since $\mathcal{O}_\varnothing=R^\Gamma\times R^\Gamma$, for $M=\varnothing $ the proof is obvious.
Suppose $M\neq\varnothing$. For $\alpha\in M$, there exist
$n_\alpha,m_\alpha\geq1$ such that
$\varphi^{m_\alpha}(\nu)=\varphi^{n_\alpha}(\alpha)$. Choose $ r\geq\mathop{\max}\limits_{\alpha\in M} n_\alpha$,
then for each $n\geq r+\mathop{\max}\limits_{\alpha\in M} m_\alpha$
and $\beta\in M$ we have:
\begin{eqnarray*}
\frac{n(n+1)}2+r-n_\beta+m_\beta &\geq  & \frac{n(n+1)}2+r-\mathop{\max}\limits_{\alpha\in M} n_\alpha+m_\beta \\
& \geq & \frac{n(n+1)}2+m_\beta >\frac{n(n+1)}2
\end{eqnarray*}
and
\begin{eqnarray*}
\frac{n(n+1)}2+r-n_\beta+m_\beta &\leq  & \frac{n(n+1)}2+r-n_\beta+\mathop{\max}\limits_{\alpha\in M} m_\alpha  \\
& \leq & \frac{n(n+1)}2-n_\beta+n \\
& < & \frac{n(n+1)}2+n<\frac{(n+1)(n+2)}2
\end{eqnarray*}
therefore $\frac{n(n+1)}2<\frac{n(n+1)}2+r-n_\beta+m_\beta<\frac{(n+1)(n+2)}2$ and:
\[\forall p\geq0\:\:\:( \frac{n(n+1)}2+r-n_\beta+m_\beta\neq\frac{p(p+1)}2)\:,\]
which shows
$x^E_{\varphi^{ \frac{n(n+1)}2+r-n_\beta+m_\beta}(\nu)}=x^F_{\varphi^{ \frac{n(n+1)}2+r-n_\beta+m_\beta}(\nu)}=0$.
Hence for $K=E,F$ we have
$x_{\varphi^{ \frac{n(n+1)}2+r}(\beta)}^K=x_{\varphi^{ \frac{n(n+1)}2+r-n_\beta+m_\beta}(\nu)}^K=0$,
which leads to
\[(\sigma_{\varphi,\mathfrak{w}}^{\frac{n(n+1)}2+r}(x^E),\sigma_{\varphi,\mathfrak{w}}^{\frac{n(n+1)}2+r}(x^F))\in\mathcal{O}_{\{\beta\}}\:.\]
Thus
{\small \begin{eqnarray*}
\{\frac{n(n+1)}2+r:n\geq r+\mathop{\max}\limits_{\alpha\in M} m_\alpha\} & \subseteq &
\mathop{\bigcap}\limits_{\beta\in M}\{n\geq1:(\sigma_{\varphi,\mathfrak{w}}^n(x^E),\sigma_{\varphi,\mathfrak{w}}^n(x^F))\in \mathcal{O}_{\{\beta\}}\} \\
& = & \{n\geq1:(\sigma_{\varphi,\mathfrak{w}}^n(x^E),\sigma_{\varphi,\mathfrak{w}}^n(x^F))\in \mathcal{O}_M\}=T_M
\end{eqnarray*}}
and $T_M$ is infinite.
\end{proof}
\begin{lemma}\label{ABC30}
For all $E,F\subseteq{\mathbb N}\cup\{0\}$, we have
$(x^E,x^F)\in Prox(\sigma_{\varphi,\mathfrak{w}})$.
\end{lemma}
\begin{proof}
For $H\subseteq \Gamma$ let
$T_H:= \{n\geq1:(\sigma_{\varphi,\mathfrak{w}}^n(x^E),\sigma_{\varphi,\mathfrak{w}}^n(x^F))\in \mathcal{O}_H\}$.
For each
$\alpha\in\Gamma\setminus\frac{\theta}{\thicksim_\varphi}$ and
each $n\geq0$, we have
$\varphi^n(\alpha)\in\Gamma\setminus\frac{\theta}{\thicksim_\varphi}$
and $x^E_{\varphi^n(\alpha)}=x^F_{\varphi^n(\alpha)}=0$, thus
$(\sigma_{\varphi,{\mathfrak w}}^n(x^E),\sigma_{\varphi,{\mathfrak
w}}^n(x^F))\in\mathcal{O}_{\{\alpha\}}$.
Which shows
\[\forall H\subseteq \Gamma\setminus\frac{\theta}{\thicksim_\varphi}\:\:\: (T_{H}=\mathbb{N})\:.\]
Thus for each finite subset $L$ of $\Gamma$ we have:
\[T_L = T_{L\cap\frac{\nu}{\thicksim_\varphi}}\cap T_{L\setminus\frac{\theta}{\thicksim_\varphi}}
=T_{L\cap\frac{\nu}{\thicksim_\varphi}}\cap\mathbb{N}=T_{L\cap\frac{\nu}{\thicksim_\varphi}}\]
By Lemma~\ref{ABC20}, $T_{L\cap\frac{\theta}{\thicksim_\varphi}}$ is infinite, thus $T_L$ is infinite, which leads to the desired result
\end{proof}
\subsection{Main theorem on Li--Yorke chaoticity of $(R^\Gamma,\sigma_{\varphi,\mathfrak{w}})$}
\noindent Now we are ready to establish our main theorem on Li--Yorke chaoticity of weighted
generalized shift $(R^\Gamma,\sigma_{\varphi,\mathfrak{w}})$.
\begin{theorem}\label{salam-yorke-30}
In finite field $R$ the following statements are equivalent (see~\cite{gen-li} for countable $\Gamma$ and generalized shift dynamical system $(\sigma_\varphi,R^\Gamma)$ too):
\\
1. $(R^\Gamma,\sigma_{\varphi,\mathfrak{w}})$ is (strongly) sensitive,
\\
2. there exists non--quasi--periodic point $\theta\in\Gamma$ such that for all $n\geq0$,
    $\mathfrak{w}_{\varphi^n(\theta)}\neq0$,
\\
3. $(R^\Gamma,\sigma_{\varphi,\mathfrak{w}})$ has at least one scrambled pair,
\\
4. $(R^\Gamma,\sigma_{\varphi,\mathfrak{w}})$ is Li--Yorke chaotic.
\end{theorem}
\begin{proof}
``${\bf (1\Leftrightarrow2)}$'' By Corollary~\ref{salam-sen-30}, (1) and (2) are equivalent.
\\
``${\bf (4\Rightarrow3)}$'' It's clear that (4) implies (3).
\\
``${\bf (3\Rightarrow2)}$'' Note that all elements of $Prox(\sigma_{\varphi,\mathfrak{w}})\setminus
Asym(\sigma_{\varphi,\mathfrak{w}})$ are scrambled pairs, thus (3) is equivalent to
$Prox(\sigma_{\varphi,\mathfrak{w}})\not\subseteq
Asym(\sigma_{\varphi,\mathfrak{w}})$, which implies (2) by Lemma~\ref{salam-yorke-10}.
\\
 ``${\bf (2\Rightarrow4)}$'' Suppose $\nu$ is a
 non--quasi--periodic point of $\varphi$ such that for all $n\geq0$,
    $\mathfrak{w}_{\varphi^n(\nu)}\neq0$.
By Remark~\ref{salam-yorke-20} there exists uncountable collection $\mathcal E$ of infinite subsets of $\mathbb N$
such that for all distinct $C,D\in\mathcal{E}$, $C\cap D$ is finite. Using notations of
Subsection~\ref{ABC}, let $Y:=\{x^A:A\in\mathcal{E}\}$,
then by Lemmas~\ref{ABC10},~\ref{ABC30},
 for each distinct $C,D\in\mathcal{E}$ we have $(x^C,x^D)\in Prox(\sigma_{\varphi,\mathfrak{w}})\setminus
Asym(\sigma_{\varphi,\mathfrak{w}})$, i.e., $x^C,x^D$ are scrambled, hence $Y$ is an uncountable scrambled
subset of $R^\Gamma$ and $(R^\Gamma,\sigma_{\varphi,\mathfrak{w}})$ is Li--Yorke chaotic.
\end{proof}
\section{When does $(R^\Gamma,\sigma_{\varphi,{\mathfrak w}})$ have dense periodic points?}
\noindent The following theorem is the main and unique theorem of this section.
\begin{theorem}\label{salam-dense-10}
For $\sigma_{\varphi,{\mathfrak w}}:R^\Gamma\to R^\Gamma$ the following statements are equivalent:
\begin{itemize}
\item[1.] $\sigma_{\varphi,{\mathfrak w}}:R^\Gamma\to R^\Gamma$ is onto,
\item[2.] $\varphi:\Gamma\to\Gamma$ is one--to--one and for all $\alpha\in\Gamma$, $\mathfrak{w}_\alpha$ is invertible,
\item[3.] $Per(\sigma_{\varphi,{\mathfrak w}})$ is dense in $R^\Gamma$.
\end{itemize}
\end{theorem}
\begin{proof}
``${\bf (1\Rightarrow2)}$'': Suppose $\sigma_{\varphi,{\mathfrak
w}}:R^\Gamma\to R^\Gamma$ is onto. There exists
$(x_\alpha)_{\alpha\in\Gamma}\in R^\Gamma$ such that
$(1)_{\alpha\in\Gamma}=\sigma_{\varphi,{\mathfrak w}}
((x_\alpha)_{\alpha\in\Gamma})=({\mathfrak w}_\alpha
x_\alpha)_{\alpha\in\Gamma}$, hence for all $\alpha\in\Gamma$ we
have ${\mathfrak w}_\alpha x_\alpha=1$ and ${\mathfrak w}_\alpha$
is invertible.
\\
For distinct $\theta,\lambda\in\Gamma$ choose
$(z_\alpha)_{\alpha\in\Gamma}\in R^\Gamma$ with $z_\theta=1$ and $z_\lambda=0$. There exists
$(y_\alpha)_{\alpha\in\Gamma}\in R^\Gamma$ with $(z_\alpha)_{\alpha\in\Gamma}=
\sigma_{\varphi,{\mathfrak w}}((y_\alpha)_{\alpha\in\Gamma})=(\mathfrak{w}_\alpha y_\alpha)_{\alpha\in\Gamma}$,
hence $1=z_\theta=\mathfrak{w}_\theta y_{\varphi(\theta)}$ and $0=z_\lambda=\mathfrak{w}_\lambda y_{\varphi(\lambda)}$
which leads to $y_{\varphi(\theta)}\neq0$ and $y_{\varphi(\lambda)}=0$, hence $y_{\varphi(\theta)}\neq y_{\varphi(\lambda)}$
and $\varphi(\theta)\neq\varphi(\lambda)$. Therefore $\varphi:\Gamma\to\Gamma$ is one--to--one.
\\
``${\bf (2\Rightarrow1)}$'': Suppose (2) is valid. Consider $(x_\alpha)_{\alpha\in\Gamma}\in R^\Gamma$ and let:
\[z_\alpha:=\left\{\begin{array}{lc} 0 & \alpha\notin\varphi(\Gamma)\:, \\ {\mathfrak w}_\beta^{-1}x_\beta &
\beta\in\Gamma,\alpha=\varphi(\beta)\:,\end{array}\right.\]
then $\sigma_{\varphi,{\mathfrak w}}((z_\alpha)_{\alpha\in\Gamma})=(x_\alpha)_{\alpha\in\Gamma}$.
\\
``${\bf (3\Rightarrow2)}$'': $Per(\sigma_{\varphi,{\mathfrak w}})$ is dense in $R^\Gamma$.
Choose distinct $\theta,\lambda\in\Gamma$ and let:
\[V_\alpha=\left\{\begin{array}{lc} \{0\} & \alpha=\lambda\:, \\ \{1\} & \alpha=\theta\:, \\
    R & \alpha\neq\theta,\lambda\:,\end{array}\right.\]
then $V:=\mathop{\prod}\limits_{\alpha\in\Gamma}V_\alpha$ is a nonempty open subset of $R^\Gamma$, hence by
hypothesis (3) there exists $(u_\alpha)_{\alpha\in\Gamma}\in Per(\sigma_{\varphi,{\mathfrak w}})\cap V$. There exists
$n\geq 1$ with $(u_\alpha)_{\alpha\in\Gamma}=\sigma_{\varphi,{\mathfrak w}}^n((u_\alpha)_{\alpha\in\Gamma})=
({\mathfrak w}_\alpha{\mathfrak w}_{\varphi(\alpha)}\cdots{\mathfrak w}_{\varphi^{n-1}(\alpha)}u_{\varphi^n(\alpha)})_{
\alpha\in\Gamma}$ which leads to
\[{\mathfrak w}_\theta{\mathfrak w}_{\varphi(\theta)}\cdots{\mathfrak w}_{\varphi^{n-1}(\theta)}u_{\varphi^n(\theta)}=1\:,
\tag{*}\]
\[{\mathfrak w}_\lambda{\mathfrak w}_{\varphi(\lambda)}\cdots{\mathfrak w}_{\varphi^{n-1}(\lambda)}u_{\varphi^n(\lambda)}
=0\: . \tag{**}\]
Using (*), ${\mathfrak w}_\theta$ is invertible and $u_{\varphi^n(\theta)}\neq 0$, however by a similar method for all
$\alpha\in\Gamma$, ${\mathfrak w}_\alpha$ is invertible. Since for all $\alpha\in\Gamma$,  ${\mathfrak w}_\alpha$ is invertible
by (**), $u_{\varphi^n(\lambda)}=0$. Therefore $u_{\varphi^n(\lambda)}=0\neq u_{\varphi^n(\theta)}$ which leads to
$\varphi(\lambda)\neq\varphi(\theta)$ and $\varphi:\Gamma\to\Gamma$ is one--to--one.
\\
``${\bf (2\Rightarrow3)}$'': Suppose (2) is valid. If $\alpha_1,\ldots,\alpha_p\in\Gamma$ are distinct and $r_1,\ldots,r_p\in R$
we should prove
$(\mathop{\bigcap}\limits_{1\leq i\leq p}U(\alpha_i,r_i))\cap Per(\sigma_{\varphi,{\mathfrak w}})\neq\varnothing$.
For $\beta\in\{\alpha_1,\ldots,\alpha_p\}$, we have following cases (see Fig. 1):
\begin{itemize}
\item[case a.] if  $\frac{\beta}{\thicksim_\varphi}$ is finite, let $A_\beta=\frac{\beta}{\thicksim_\varphi}$ (i.e., $\beta$ is periodic),
\item[case b.] if there exists $\alpha\in\Gamma$ with infinite $\frac{\beta}{\thicksim_\varphi}=\{\varphi^n(\alpha):n\geq0\}$
    and $\beta=\varphi^m(\alpha)$ for $m\geq0$, let $A_\beta=\{\varphi^n(\alpha):0\leq n\leq m\}$,
\item[case c.] if for all $n\in\mathbb{Z}$  we have $\varphi^n(\beta)\neq\varnothing$, suppose $\frac{\beta}{\thicksim_\varphi}\cap
    \{\alpha_1,\ldots,\alpha_p\}$ is equal to $\{\varphi^{t_1}(\beta),\ldots,\varphi^{t_s}(\beta)\}$ with $t_1<t_2<\cdots<t_s$ let
    $A_\beta=\{\varphi^i(\beta):t_1\leq i\leq t_s\}$.
\end{itemize}
Let $A=\bigcup\{A_{\alpha_i}:1\leq i\leq p\}$, then $A$ is a finite subset of $\Gamma$. Choose arbitrary
$(x_\alpha)_{\alpha\in\Gamma}\in \mathop{\bigcap}\limits_{1\leq i\leq p}U(\alpha_i,r_i)$. For $\alpha\in\Gamma$
we have the following cases:
\begin{itemize}
\item if $\alpha\in\Gamma\setminus\bigcup\{\frac{\beta}{\thicksim_\varphi}:\beta\in A\}$, then let $y_\alpha=0$,
\item if there exists periodic point $\theta\in A$ with $\alpha\in \frac{\theta}{\thicksim_\varphi}$, then $\alpha\in A_\theta\subseteq A$,
    let $y_\alpha:=x_\alpha$,
\item if there exists non--periodic point $\theta\in A$ with
    $\alpha\in \frac{\theta}{\thicksim_\varphi}=\{\varphi^m(\theta):m\geq0\}$, then there exists $n\geq1$ with
    $\theta,\varphi(\theta),\ldots,\varphi^{n-1}(\theta)\in A$ and $\varphi^{n}(\theta),\varphi^{n+1}(\theta),\ldots\notin
    A$. Let $y_{\varphi^i(\theta)}=x_{\varphi^i(\theta)}$ for $i=0,\ldots,n-1$ and
    \[y_{\varphi^{n+j}(\theta)}=(\mathfrak{w}_{\varphi^j(\theta)}\mathfrak{w}_{\varphi^{j+1}(\theta)}\cdots\mathfrak{w}_{\varphi^{j+n-1}(\theta)})^{-1}y_{\varphi^{j}(\theta)}\tag{+}\]
for all $j\geq0$.
\item if there exists non--periodic point $\mu\in A$ with $\varphi^m(\mu)\neq\varnothing$ for all $m\in{\mathbb Z}$ and
    $\alpha\in \frac{\mu}{\thicksim_\varphi}=\{\varphi^m(\mu):m\in{\mathbb Z}\}$, then there exists
    $\theta\in\frac{\mu}{\thicksim_\varphi}$ and $n\geq1$ with
    $\theta,\varphi(\theta),\ldots,\varphi^{n-1}(\theta)\in A$ and $\varphi^{n}(\theta),\varphi^{n+1}(\theta),\ldots\notin A$,
    also $\varphi^{-1}(\theta),\varphi^{-2}(\theta),\ldots\notin A$.
    Let $y_{\varphi^i(\theta)}=x_{\varphi^i(\theta)}$ for $i=0,\ldots,n-1$ and (+)
    for all $j\in\mathbb{Z}$.
\end{itemize}
For $\theta\in A$, let
$n_\theta=|\frac{\theta}{\thicksim_\varphi}\cap A|$ and
$T=\:|\{r\in R\setminus\{0\}:r$ is invertible
$\}|\:\mathop{\prod}\limits_{\theta\in A}n_\theta$, then
$y=(y_\alpha)_{\alpha\in\Gamma}\in \mathop{\bigcap}\limits_{1\leq
i\leq p}U(\alpha_i,r_i)$ and
$\sigma_{\varphi,\mathfrak{w}}^T(y)=y$, since for all
$\alpha\in\Gamma$, we have the following cases:
\begin{itemize}
\item $\alpha\notin \bigcup\{\frac{\beta}{\thicksim_\varphi}:\beta\in A\}$. In this case for all $i$, we have
    $\varphi^i(\alpha)\notin\bigcup\{\frac{\beta}{\thicksim_\varphi}:\beta\in A\}$, hence $y_\alpha=y_{\varphi^i(\alpha)}=
    y_{\varphi^T(\alpha)}=0$, hence:
    \[y_\alpha=0=\mathfrak{w}_\alpha\mathfrak{w}_{\varphi(\alpha)}\cdots\mathfrak{w}_{\varphi^{T-1}(\alpha)}y_{\varphi^{T}(\alpha)}\]
\item $\alpha\in \bigcup\{\frac{\beta}{\thicksim_\varphi}:\beta\in A\}$ is not periodic,
    in this case there exists non--periodic point $\theta\in A$ and $j$ such that $\alpha\in\frac{\theta}{\thicksim_\varphi}$,
    $\varphi^j(\theta)=\alpha$ and (+) is valid for $n=n_\theta$, so (+) is valid for all multiplications of $n_\theta$
    like $T$, therefore
    \[y_{\varphi^j(\theta)}=\mathfrak{w}_{\varphi^j(\theta)}\mathfrak{w}_{\varphi^{j+1}(\theta)}\cdots\mathfrak{w}_{\varphi^{T+j-1}(\theta)}y_{\varphi^{T+j}(\theta)}\]
    hence:
    \[y_\alpha=\mathfrak{w}_\alpha\mathfrak{w}_{\varphi(\alpha)}\cdots\mathfrak{w}_{\varphi^{T-1}(\alpha)}y_{\varphi^{T}(\alpha)}\tag{++}\]
\item $\alpha\in \bigcup\{\frac{\beta}{\thicksim_\varphi}:\beta\in A\}$ is periodic,
    in this case there exists a periodic point $\theta\in A$ such that $\alpha\in\frac{\beta}{\thicksim_\varphi}\subseteq A$.
    Moreover $\varphi^{n_\alpha}(\alpha)=\alpha$, so for $k=\:|\{r\in R\setminus\{0\}:r$ is invertible$\}|$
    we have $\varphi^{kn_\alpha}(\alpha)=\alpha$ and for all invertible elements $r\in R\setminus\{0\}$, $r^k=1$, in
    particular $\mathfrak{w}_\psi^k=1$ for all $\psi\in\Gamma$, hence we have:
    \begin{eqnarray*}
    \mathfrak{w}_\alpha\mathfrak{w}_{\varphi(\alpha)}\cdots\mathfrak{w}_{\varphi^{kn_\alpha-1}(\alpha)}y_{\varphi^{kn_\alpha}(\alpha)} & = & \mathfrak{w}_\alpha\mathfrak{w}_{\varphi(\alpha)}\cdots\mathfrak{w}_{\varphi^{kn_\alpha-1}(\alpha)}y_\alpha \\
    & = & (\mathfrak{w}_\alpha\mathfrak{w}_{\varphi(\alpha)}\cdots\mathfrak{w}_{\varphi^{n_\alpha-1}(\alpha)})^ky_\alpha \\
    & = & \mathfrak{w}_\alpha^k\mathfrak{w}_{\varphi(\alpha)}^k\cdots\mathfrak{w}_{\varphi^{n_\alpha-1}(\alpha)}^ky_\alpha \\
    & = & y_\alpha
    \end{eqnarray*}
and $\mathfrak{w}_\alpha\mathfrak{w}_{\varphi(\alpha)}\cdots\mathfrak{w}_{\varphi^{m-1}(\alpha)}y_{\varphi^{m}(\alpha)}=y_\alpha$ for $m=kn_\alpha$ and all of multiplations of $kn_\alpha$ like $T$, hence again we have (++).
\end{itemize}
Using above cases (++) is valid for all $\alpha\in\Gamma$ and:
\[\sigma^T_{\varphi,\mathfrak{w}}((y_\alpha)_{\alpha\in\Gamma})=(\mathfrak{w}_\alpha\mathfrak{w}_{\varphi(\alpha)}\cdots\mathfrak{w}_{\varphi^{T-1}(\alpha)}y_{\varphi^{T}(\alpha)})_{\alpha\in\Gamma}=(y_\alpha)_{\alpha\in\Gamma}\]
Therefore $\mathop{\bigcap}\limits_{1\leq i\leq p}U(\alpha_i,r_i)\cap Per(\sigma_{\varphi,\mathfrak{w}})\neq\varnothing$.
\end{proof}
\begin{center}
\begin{tabular}{|c|c|}
& case \\
\hline
{\tiny $\underbrace{\xymatrix{\varphi(\beta)\ar[r] &\varphi^2(\beta)\ar[r] & \cdots\ar[r] & \varphi^j(\beta)\ar[dlll] \\
\beta=\varphi^{j+1}(\beta) \ar[u] &&& }}_{\frac{\beta}{\thicksim_\varphi}=A_\beta}$}
& a \\ \hline
{\tiny $\underbrace{\xymatrix{\alpha\ar[r]\ar@{-}[d] &\varphi(\alpha)\ar[r] & \cdots\ar[r] & \varphi^m(\alpha)=\beta\ar[r]\ar@{-}[d]&\varphi^{m+1}(\alpha)\ar[r] & \cdots \\
\ar@{-}[r] & A_\beta \ar@{-}[rr]&  &&&}}_{\frac{\beta}{\thicksim_\varphi} }$}
& b \\ \hline
{\tiny $\underbrace{\xymatrix{\cdots\ar[r] & \varphi^{t_1-1}(\beta)\ar[r] & \varphi^{t_1}(\beta)\ar[r]\ar@{-}[d] & \varphi^{t_1+1}(\beta)\ar[r] &\cdots\ar[r]  & \varphi^{t_s}(\beta)\ar[r]\ar@{-}[d]&\varphi^{t_s+1}(\beta)\ar[r] & \cdots \\
& & \ar@{-}[r] & A_\beta \ar@{-}[rr]&  &&&}}_{\frac{\beta}{\thicksim_\varphi} }$}
& c \\ \hline
\end{tabular} \\ $\:$  \\ Fig. 1 \end{center}
\section{Devaney chaotic $(R^\Gamma,\sigma_{\varphi,{\mathfrak w}})$, for finite field $R$}
\noindent We say the dynamical system $(X,f)$ is topological
transitive if for each nonempty open subsets $U,V$ of $X$, there
exists $n\geq1$
    with $f^n(U)\cap V\neq\varnothing$. Moreover
$(X,f)$ with compact Hausdorff $X$  is Devaney chaotic if it is sensitive, topological transitive, and $Per(f)$ is dense in X.
\\
In this section we prove that
 $(\sigma_{\varphi,\mathfrak{w}},R^\Gamma)$ is Devaney chaotic (topological transitive) if and only if
$\varphi$ is one--to--one without any periodic point and $\mathfrak{w}_\alpha$ is invertible for each $\alpha\in\Gamma$.
\begin{lemma}\label{salam-tran-10}
Weighted generalized shift $(R^\Gamma,\sigma_{\varphi,\mathfrak{w}})$ is topological transitive
if and only if $\varphi$ is one--to--one without periodic points and for all $\alpha\in\Gamma$,
$\mathfrak{w}_\alpha$ is invertible.
\end{lemma}
\begin{proof}
``$\Rightarrow$'' Suppose
$(R^\Gamma,\sigma_{\varphi,\mathfrak{w}})$ is topological
transitive, for each nonempty open subset $U$ of $R^\Gamma$,
there exists $n\geq1$ such that
$U\cap\sigma_{\varphi,\mathfrak{w}}^n (R^\Gamma)\neq
\varnothing$, thus $U\cap
\sigma_{\varphi,\mathfrak{w}}(R^\Gamma)\neq\varnothing$. Hence
$\sigma_{\varphi,\mathfrak{w}}(R^\Gamma)$ is dense in $R^\Gamma$.
Since $R^\Gamma$ is compact Hausdorff and
$\sigma_{\varphi,\mathfrak{w}}:R^\Gamma\to R^\Gamma$ is
continuous, $\sigma_{\varphi,\mathfrak{w}}(R^\Gamma)$ is a closed
subset of $R^\Gamma$ which leads to
$\sigma_{\varphi,\mathfrak{w}}(R^\Gamma)=\overline{
\sigma_{\varphi,\mathfrak{w}}(R^\Gamma)}=R^\Gamma$. Therefore
$\sigma_{\varphi,\mathfrak{w}}:R^\Gamma\to R^\Gamma$ is onto and
by Theorem~\ref{salam-dense-10}, $\varphi$ is one--to--one and
all $\mathfrak{w}_\alpha$ s are invertible.
\\
We prove $\varphi$ does not have any periodic point. Suppose $\theta\in Per(\varphi)$, consider
following nonempty open subsets of $R^\Gamma$ (note that $\{\varphi^n(\theta):n\geq0\}$ is a finite subset of $\Gamma$):
\[U:=\mathop{\bigcap}\limits_{\alpha\in\{\varphi^n(\theta):n\geq0\}}U(\alpha,1)\:,\:
V:=\mathop{\bigcap}\limits_{\alpha\in\{\varphi^n(\theta):n\geq0\}}U(\alpha,0)\:.\]
For all $n\geq1$ and $(x_\alpha)_{\alpha\in\Gamma}$, we have:
\begin{eqnarray*}
(x_\alpha)_{\alpha\in\Gamma}\in V
    & \mathop{\Rightarrow}\limits^{V\subseteq U(\varphi^n(\theta),0)}
    & (x_\alpha)_{\alpha\in\Gamma}\in U(\varphi^n(\theta),0) \\
& \Rightarrow & x_{\varphi^n(\theta)}=0 \\
& \Rightarrow &
    \mathfrak{w}_\theta\mathfrak{w}_{\varphi(\theta)}\cdots\mathfrak{w}_{\varphi^{n-1}(\theta)}x_{\varphi^n(\theta)}=0 \\
& \Rightarrow &
    (\mathfrak{w}_\alpha\mathfrak{w}_{\varphi(\alpha)}\cdots\mathfrak{w}_{\varphi^{n-1}(\alpha)}x_{\varphi^n(\alpha)})_{
    \alpha\in\Gamma}\notin U(\theta,1) \\
& \Rightarrow & \sigma^n_{\varphi,\mathfrak{w}}((x_\alpha)_{\alpha\in\Gamma})\notin U(\theta,1) \\
& \mathop{\Rightarrow}\limits^{U(\theta,1)\supseteq U} &
    \sigma^n_{\varphi,\mathfrak{w}}((x_\alpha)_{\alpha\in\Gamma})\notin U
\end{eqnarray*}
therefore $\sigma^n_{\varphi,\mathfrak{w}}(V)\cap U=\varnothing$ for all $n\geq0$ which is in contradiction with
topological transitivity of $(R^\Gamma,\sigma_{\varphi,\mathfrak{w}})$, hence $Per(\varphi)=\varnothing$.
\\
``$\Leftarrow$'' Suppose $\varphi$ is one--to--one without any periodic point and $\mathfrak{w}_\alpha$ s are invertible.
Suppose $U,V$ are nonempty open subsets of $R^\Gamma$, then:
\\
$\bullet$ there exist $r_1,\ldots,r_n\in R$ and distinct $\alpha_1,\ldots,\alpha_n\in\Gamma$ with
    $\mathop{\bigcap}\limits_{i=1}^n U(\alpha_i,r_i)\subseteq U$,
\\
$\bullet$ there exist $s_1,\ldots,s_m\in R$ and distinct $\beta_1,\ldots,\beta_m\in\Gamma$ with
    $\mathop{\bigcap}\limits_{i=1}^m U(\beta_i,s_i)\subseteq V$.
\\
We may suppose $\mathop{\bigcap}\limits_{i=1}^n U(\alpha_i,r_i)=\mathop{\prod}\limits_{\alpha\in\Gamma}U_\alpha$
(so for $\alpha\neq\alpha_1,\ldots,\alpha_n$ we have $U_\alpha=R$)
and $\mathop{\bigcap}\limits_{i=1}^m U(\beta_i,s_i)=\mathop{\prod}\limits_{\alpha\in\Gamma}V_\alpha$
(so for $\alpha\neq\beta_1,\ldots,\beta_m$ we have $V_\alpha=R$).
For each $\alpha\in\Gamma$, $\{\varphi^n(\alpha)\}_{n\geq1}$ is a one--to--one sequence, hence there exists
$k_\alpha\geq1$ such that
\[\{\varphi^n(\alpha):n\geq k_\alpha\}\cap\{\alpha_1,\ldots,\alpha_n\}=\varnothing\:.\]
Let $N=\max(k_{\beta_1},\cdots,k_{\beta_m})$, then for each
$p\geq N$ and $\alpha\in\Gamma$, we have the following cases:
\begin{itemize}
\item[i.] $\alpha\neq\beta_1,\ldots,\beta_m$. In this case $V_\alpha=R$, hence:
\begin{eqnarray*}
(\mathfrak{w}_\alpha\cdots\mathfrak{w}_{\varphi^{p-1}(\alpha)}
    U_{\varphi^p(\alpha)})\cap V_\alpha& = &(\mathfrak{w}_\alpha\cdots\mathfrak{w}_{\varphi^{p-1}(\alpha)}
    U_{\varphi^p(\alpha)})\cap R \\
&=& \mathfrak{w}_\alpha\cdots\mathfrak{w}_{\varphi^{p-1}(\alpha)}
    U_{\varphi^p(\alpha)}\neq\varnothing\:.
\end{eqnarray*}
\item[ii.] there exists $j\in\{1,\ldots,m\}$ such that $\alpha=\beta_j$. In this case
    using $p\geq N\geq k_{\beta_j}$     we have $\varphi^p(\alpha)=\varphi^p(\beta_j)\neq\alpha_1,\ldots,\alpha_n$,
    hence $U_{\varphi^p(\alpha)}=R$, therefore (note that for all $\gamma\in\Gamma$, $\mathfrak{w}_\gamma$ is
    invertible which leads to $\mathfrak{w}_\gamma R=R$):
\begin{eqnarray*}
(\mathfrak{w}_\alpha\cdots\mathfrak{w}_{\varphi^{p-1}(\alpha)}
    U_{\varphi^p(\alpha)})\cap V_\alpha & = &
    (\mathfrak{w}_\alpha\cdots\mathfrak{w}_{\varphi^{p-1}(\alpha)}
    R)\cap V_\alpha \\
& = & R\cap V_\alpha=V_\alpha\neq\varnothing\:.
\end{eqnarray*}
\end{itemize}
By the above cases:
\[\forall\alpha\in\Gamma\:\:( (\mathfrak{w}_\alpha\cdots\mathfrak{w}_{\varphi^{p-1}(\alpha)}
    U_{\varphi^p(\alpha)})\cap V_\alpha\neq\varnothing)\:,\]
thus:
\begin{eqnarray*}
\sigma^p_{\varphi,\mathfrak{w}}(U)\cap V & \supseteq &
    \sigma^p_{\varphi,\mathfrak{w}}(\mathop{\bigcap}\limits_{i=1}^n U(\alpha_i,r_i))\cap
    \mathop{\bigcap}\limits_{i=1}^m U(\beta_i,s_i) \\
& = & \sigma^p_{\varphi,\mathfrak{w}}(\mathop{\prod}\limits_{\alpha\in\Gamma}U_\alpha)\cap
    (\mathop{\prod}\limits_{\alpha\in\Gamma}V_\alpha) \\
& = & (\mathop{\prod}\limits_{\alpha\in\Gamma}\mathfrak{w}_\alpha\cdots\mathfrak{w}_{\varphi^{p-1}(\alpha)}
    U_{\varphi^p(\alpha)})\cap (\mathop{\prod}\limits_{\alpha\in\Gamma}V_\alpha) \\
& = & \mathop{\prod}\limits_{\alpha\in\Gamma}((\mathfrak{w}_\alpha\cdots\mathfrak{w}_{\varphi^{p-1}(\alpha)}
    U_{\varphi^p(\alpha)})\cap V_\alpha) \neq\varnothing\:.
\end{eqnarray*}
Hence $\sigma^p_{\varphi,\mathfrak{w}}(U)\cap V\neq\varnothing$, which leads to transitivity of
$\sigma_{\varphi,\mathfrak{w}}:R^\Gamma\to R^\Gamma$.
\end{proof}
\begin{theorem}\label{salam-tran 20}
The following statements are equivalent (see~\cite{dev} for countable $\Gamma$ and generalized shift dynamical system $(\sigma_\varphi,R^\Gamma)$ too):
\begin{itemize}
\item $(\sigma_{\varphi,\mathfrak{w}},R^\Gamma)$ is Devaney chaotic,
\item $(R^\Gamma,\sigma_{\varphi,\mathfrak{w}})$ is topological transitive,
\item $\varphi$ is one--to--one without any periodic point and $\mathfrak{w}_\alpha$ is invertible for each $\alpha\in\Gamma$.
\end{itemize}
\end{theorem}
\begin{proof}
Use Theorem~\ref{salam-dense-10}, Lemma~\ref{salam-tran-10} and Lemma~\ref{salam-sen-20}.
\end{proof}
\begin{note}
One may obtain Theorem~\ref{salam-tran 20} by Theorem~\ref{salam-dense-10}, Lemma~\ref{salam-tran-10} and~\cite[Theorem 4.7]{jaleb} too.
\end{note}
\subsection*{Acknowledgement} The fourth author was partly
supported by a grant from INSF (No. 95820471).

\noindent {\small {\bf Fatemah Ayatollah Zadeh Shirazi}, Faculty of
Mathematics, Statistics and Computer Science, College of Science,
University of Tehran, Enghelab Ave., Tehran, Iran \linebreak
({\it e-mail}: f.a.z.shirazi@ut.ac.ir, fatemah@khayam.ut.ac.ir)}
\\
{\small {\bf Elaheh Hakimi}, School of Mathematics, Statistics
and Computer Science, College of Science, University of Tehran,
Enghelab Ave., Tehran, Iran ({\it e-mail}:
elaheh.hakimi@gmail.com)}
\\
{\small {\bf Arezoo Hosseini},
Faculty of Mathematics, College of Science, Farhangian University, Pardis Nasibe--shahid sherafat, Enghelab Ave., Tehran, Iran
({\it e-mail}: a.hosseini@cfu.ac.ir)}
\\
{\small {\bf Reza Rezavand}, School of Mathematics, Statistics and
Computer Science, College of Science, University of Tehran,
Enghelab Ave., Tehran, Iran ({\it e-mail}: rezavand@ut.ac.ir)}

\end{document}